\documentclass{article}
\usepackage[utf8]{inputenc}
\usepackage[english]{babel}
\usepackage{amsmath}
\usepackage{amssymb}
\usepackage{amsthm}
\usepackage{enumerate}
\usepackage{color,graphicx}
\usepackage{lastpage209}

\numberwithin{equation}{section}

\newtheorem{theorem}{Theorem}[section] % 1st argument is your name for it
\newtheorem{lemma}[theorem]{Lemma}     % 2nd argument is what is printed
\newtheorem{corollary}[theorem]{Corollary}
\newtheorem{proposition}[theorem]{Proposition}
\theoremstyle{definition}
\newtheorem{definition}[theorem]{Definition}
\newtheorem{remark}[theorem]{Remark}

\newtheorem{example}[theorem]{Example}

\def\F{{\mathbb F}}
\def\NN{{\mathbb N}}

\def\A{\mathcal A}
\def\B{\mathcal B}
\def\L{\mathcal L}

\def\RR{\mathcal R}

\def\SS{\mathcal S}

\def\dim{{\rm dim}\,}

\def\R{{\mathbb R}}

\def\C{{\mathbb{C}}}

\begin{document}
\title%[Upper bounds for the lengths]
{Upper bounds for the length  of non-associative algebras\thanks{The work is  financially supported by the grant  RSF 17-11-01124 }}

\author{A.E. Guterman$^{1,2}$, D.K. Kudryavtsev$^{1}$}
\date{\small $^{1}$Faculty of Algebra, Department of Mathematics and 	Mechanics, 	Moscow State University,  Moscow, GSP-1, 119991, 	Russia\\\small $^{2}$Moscow Institute of Physics and Technology,  Dolgoprudny, 141701, Russia}

\maketitle

\begin{abstract}
 
We obtain a sharp upper bound for the length of arbitrary   non-associative  
algebra and present an example demonstrating the sharpness of our bound. 
To show this we introduce a new method of characteristic sequences based on linear algebra technique. This method provides an efficient tool for computing the length function in non-associative case. Then we apply the introduced method to obtain an upper bound for the length of an arbitrary locally complex algebra. We also show that the obtained bound is sharp. In the last case the length is bounded in terms of Fibonacci sequence.


\end{abstract}

\section{Introduction}

In the present paper $\A$ is a unital finite dimensional not necessarily associative algebra over a field $\F$. We refer the reader to \cite{McCr,ZhevSliSheShi83} for the background on the topic.  Let $\SS=\{a_1,\ldots,a_k\}$ be a finite subset of elements of the algebra $\A$. We define the length 
function of $\SS$ as follows.

Any product of a finite number of elements of $\SS$ is a {\em word} in letters from $\SS$, or simply a word in $\SS$. The {\em length} of the word
	equals to the  number of letters in the corresponding product. We consider $1$ as a word in $\SS$ of the {\em length $0$}.

It is worth noting that different choices of brackets provide different words of the same length due to the non-associativity of $\A$.

The set of all words  in $\SS$ with lengths less  than or equal to $i$ is denoted by $S_i$, here $i\ge 0$.
 
Note that similar to the associative case, $m<n$ implies that $S_m \subseteq S_n$.
 
The set $\L_i(\SS) = \langle S_i\rangle$  is the linear span  of  the set  $S_i$
(the set of all finite linear combinations with coefficients belonging to $\mathbb{\F}$). We write $\L_i$ instead of $\L_i(\SS)$ if $\SS$ is determined from the context. It should be noted that $\L_0(\SS)=\langle
1 \rangle=\mathbb{F}$ for any $\SS$. 
The set $\L(\SS)$ stands for $\bigcup\limits_{i=0}^\infty \L_i(\SS)$. 
\begin{remark}\label{rem}
	$\SS$ is a generating set of
	$\A$ if and only if $\A=\L(\SS)$.
\end{remark}

\begin{definition}\label{sys_len} 
	The	{\em length of  a generating set} $\SS$ of a finite-dimensional algebra $\A$ is defined as follows $l(\SS)=\min\{k\in \mathbb{Z}_+:\L_k(\SS)=\A\}.$ 
\end{definition}

\begin{definition}\label{alg_len} The
	{\em length of an algebra $\A$} is   $l(\A)=\max \{l(\SS): \L(\SS)=\A\}$. 
\end{definition}

The problem of the  associative algebra length computation was first discussed in \cite{SpeR59}, \cite{SpeR60} for the algebra of $3\times3$ matrices
in the context of the mechanics of isotropic continua.
The problem
of computing the length of the full matrix algebra $M_n(\F)$ as a function of the matrix size $n$ was stated in the work~\cite{Paz84}  and is still an open problem. The known upper bounds for the length of the matrix algebra are in general nonlinear in $n$.

 The first upper bound on the length function was established in 1984 by Paz, see \cite{Paz84}.
\begin{theorem}[{\cite[Theorem~1, Remark~1]{Paz84}}] Let $\F$ be an arbitrary field. Then \[l(M_n(\F))\leq
\left\lceil \frac{n^2+2}{3}\right\rceil,\] where $\lceil . \rceil$
denotes the least integer function.
\end{theorem}

An (asymptotic) improvement of this bound was obtained in~\cite{Pap97}.
More precisely, Pappacena in 1997
%his paper of 1997~\cite{Pap97} 
provided an upper bound for the length of any finite dimensional associative algebra $\RR$ as a function of two its invariants: the dimension and
$m(\RR)$, which is the maximal degree of the minimal polynomials for the elements
of the algebra.

\begin{theorem}[{\rm \cite[Theorem 3.1]{Pap97}}]\label{P:bound0}
Let $\F$ be an arbitrary field, $\RR$ be an associative $\F$-algebra, and let
\[f(d,m)=m\sqrt{\frac{2d}{m-1}+\frac{1}{4}}+\frac{m}{2}-2.\] Then
$ l(\RR)< f(\dim \RR,m(\RR)).$ \end{theorem}

For the matrix algebra the theorem above  provides a bound with asymptotic behavior
$\displaystyle O(n^{3/2})$.

These bounds are not sharp. However, there are sharp bounds  on the length function, which are established for certain classes of associative algebras. For example, the length of commutative matrix subalgebras of size $n$ is  bounded by~$(n-1)$, see~\cite{GutM09,Mar09SM}, where this bound was proved and in particular it was shown that the bound $(n-1)$ can be achieved on the algebra of diagonal matrices over an infinite field. 
In the recent paper \cite{GutK16} we evaluate the length function for quaternion and octonion algebras. 

In this paper we obtain an upper bound for the length of arbitrary non-associative %locally-complex 
algebra and provide an example demonstrating that our bound on the %this 
length is sharp. Namely, let $\A$ be an  $\F$-algebra,  $\dim \A = n>2$. We show that $l( \A)\le 2^{n-2}$ and provide an example of an $n$-dimensional algebra of the length exactly~$2^{n-2}$.

To show this we introduce a new method to compute the length. We call it the method of characteristic sequences. This method is based on linear algebraic technique and provides an efficient tool for computing the length function in non-associative case. Then in addition we apply our method to obtain   a sharp upper bound for the lengths of  locally-complex algebras. Here Fibonacci sequence $F_n=(0,1,1,2,3,5,8,\ldots)$ appears. Namely, we show that if $\B$ is an $n$-dimensional locally complex $\F$-algebra, then $l(\B)\le F_{n-1}$. Moreover, we demonstrate that for each $n$ there exists $n$-dimensional locally complex $\F$-algebra of the length exactly~$F_{n-1}$. This method as well as most of the results in the present paper is based on basic concepts from linear algebra such as linear span and properties of linear subspaces or their dimensions.

Observe that in associative case there are different attempts to find general methods to compute lengths of algebras and generating sets. These methods depend on the structure of algebra. Namely different methods are developed for matrix algebras containing some matrices of a special structure, group algebras of abelian or non-abelian groups, incidence algebras, see respectively \cite{GutLMS18,GutMC19,GM19,KM19}. However, nothing similar to the method of characteristic sequences was known.

We note that characteristic  sequences belong to the class of integer sequences named additive chains. These sequences are known since ancient times and had several reincarnations. The detailed and self-contained survey of this  theory still containing lots of open problems can be found in \cite[Chapter 4.6.3]{Knut}. %In particular, it allows to obtain the upper bounds for some special algebras and generating systems. Also it allows to describe what particular integers can be realized as a value of the length function of some algebra of the dimension~$n$.

In the subsequent paper \cite{GutK19b} we characterize all integer sequences that may serve as characteristic sequences for some non-associative algebras as well as characteristic sequences for locally complex algebras.

Our paper is organized as follows. In Section 2 we discuss some very general properties of the length function and differences between associative and non-associative cases. In Section 3 we introduce the characteristic sequence of a generating set of an algebra and investigate its general properties. Section 4 is devoted to establishing the upper bounds for the lengths of non-associative algebras. Section 5 reminds some basic properties of locally-complex algebras and adopts for them general results from Section 2. In Section 6 we use characteristic sequence to find the upper bounds for the lengths of locally-complex algebras and to prove its sharpness. %In Section 6 we investigate particular values of the length function and their feasibility. In particular, we prove that the obtained bounds are sharp, show that not all values which are less than the bounds are feasible %p, prove that any integer sequence possessing the properties of a characteristic sequence is a characteristic sequence for a certain appropriate algebra, 
%and investigate the bound for subsequent values of the length function.

\section{Properties of the length in non-associative case}

\begin{lemma}
	Suppose   $m,n \in \mathbb{N}$  are given such that $m<n$. Then the following statements are equivalent:
	\begin {enumerate}
	\item $\L_n(\SS)=\L_m(\SS)$,
	\item $\dim \L_n(\SS)= \dim \L_m(\SS)$.
	\end {enumerate}
\end{lemma}
\begin{proof}
	The statement follows directly from the fact that $\L_k(\SS)$ is a linear subspace of $\A$ and $\L_m(\SS) \subseteq \L_n(\SS)$ for $m<n$.
\end{proof}

\begin{lemma}\label{lem_sub}
 Let $\A$  be an algebra and $\SS_0$ and $\SS_1$  be its finite subsets such that $\L_1(\SS_0) \subseteq \L_1(\SS_1)$. Then $\L_k(\SS_0)   \subseteq \L_k(\SS_1)$ for every positive integer $k$.
\end{lemma}
\begin{proof}
	We will prove this statement by induction on $k$.
	
	The base: for $k=1$ the statement is given.
	
	The step: 	Directly from definitions we get $\L_k(\SS)=\langle \bigcup\limits_{i=1}^{k-1}\L_i(\SS)  \cdot \L_{k-i}(\SS)\rangle$ for a generating set. Let us assume that for every $k=1,\ldots,n-1$,  $\L_k(\SS_0) \subseteq \L_k(\SS_1)$. Then $\L_n(\SS_0)=\langle \bigcup\limits_{i=1}^{n-1}\L_i(\SS_0) \cdot \L_{k-i}(\SS_0) \rangle \subseteq \langle \bigcup\limits_{i=1}^{n-1}\L_i(\SS_1) \cdot \L_{k-i}(\SS_1) \rangle = \L_n(\SS_1)$.
\end{proof}

\begin{corollary}\label{lem_halfmix}
	If $\A$ is an algebra, $\SS_0$ and $\SS_1$ are generating sets  such that $\L_1(\SS_0) \subset \L_1(\SS_1)$, then $l(\SS_0) \ge l(\SS_1)$.
\end{corollary}
\begin{proof}
	By using the result above, $\L_{l(\SS_1)-1}(\SS_0) \subset \L_{l(\SS_1)-1}(\SS_1)\neq \A$, hence $l(\SS_0) \ge l(\SS_1)$.
\end{proof}

\begin{lemma}
	If $\A$ is an algebra and $\SS_0$ and $\SS_1$ are its finite subsets such that $\L_1(\SS_0)=\L_1(\SS_1)$, then $\L_k(\SS_0)= \L_k(\SS_1)$ for every natural $k$.
\end{lemma}
\begin{proof}
	Follows from Lemma \ref{lem_sub} by applying it twice.
\end{proof}

\begin{corollary}\label{lem_mix}
If $\A$ is an algebra, $\SS_0$ is its  generating set and $\SS_1$ is a finite subset such that $\L_1(\SS_0)=\L_1(\SS_1)$, then $\SS_1$ generates $\A$ and $l(\SS_0)=l(\SS_1)$.
\end{corollary}
\begin{proof}
	By using the result above, we get: $\L_{l(\SS_0)-1}(\SS_1)=\L_{l(\SS_0)-1}(\SS_0)\neq \A$, while $\L_{l(\SS_0)}(\SS_1)=\L_{l(\SS_0)}(\SS_0)= \A$, which means that $\SS_1$ is a generating set of $\A$ and its length is equal to $l(\SS_0)$.
\end{proof}

\begin{remark}\label{rem2} Note that unlike the associative case the equality $\dim \L_n(\SS)= \dim \L_{n+1}(\SS)$ for some $n\in \mathbb{N}$ may not imply that   $\dim \L_n(\SS)= \dim \L_m(\SS)$ for all $m\ge n$ as the following example shows. 
\end{remark}

\begin{example} \label{ex1} Let $\A $ be generated by $
	1, e_1, e_2, e_3$ with the multiplication rules $  e_1^2=e_2, \, e_2^2=e_3, \, e_3^2=0$, $e_ie_j=0$ for all pairs $i,j\in \{1,2,3\}, i\ne j$. Then $\SS=\{e_1\}$ is a generating system and we have $\L_1=\langle 1,e_1\rangle$,  $\L_2=\L_3=\langle 1,e_1,e_2\rangle$, and $\L_4=\langle 1,e_1,e_2,e_3\rangle=\A$, so $\dim \L_4>\dim L_3=\dim \L_2$.
\end{example}

However, the following proposition, which belongs to folklore, is true for any non-associative algebra.

\begin{proposition}\label{pr_1}
	Let us consider a finite subset $\SS$ of $\A$ and integer $n \geq 1$. If \[\dim \L_{n}(\SS) = \dim \L_{n+1}(\SS) = \ldots = \dim \L_{2n}(\SS),\] then for all $t \in \mathbb{N}$ it holds  that $\dim \L_{n}(\SS) = \dim \L_{n+t}(\SS)$.
\end{proposition}

\begin{proof}
	We prove this statement using the  induction on $t$. The base is  $t \leq n$. In this case the assertion holds   by the conditions.
	
	Suppose we have proven the statement for all $t \leq n+k$, $k\ge 0$. Let us show that it is satisfied for $t=n+k+1$. 
	
	If $s$ is a word of the  length $n+k+1$, then it can be represented as a product of two words of  smaller non-zero lengths, $s=(s_1)(s_2)$. Both  these words are elements of $\L_{n}(\SS)$. Indeed,  if the length of a word is less than or equal to $n$, then it is an element of  $\L_{n}(\SS)$ by the definition of $\L_n(\SS)$. If the length of a word is greater than $n$, but strictly  less than $n+k+1$, then the inclusion follows from the induction hypothesis. In any case $s_1,s_2\in \L_n(\SS)$.   Hence, \[s\in \L_{n}(\SS)\cdot \L_{n}(\SS)\subseteq \L_{2n}(\SS)=\L_{n}(\SS).\] This implies $\L_{n+k+1}(\SS)=\L_{n}(\SS)$, because the linear space $\L_{n+k+1}(\SS)$  is generated by all words of length less than or equal to~$n+k+1$, which concludes the induction proof.
\end{proof}

\begin{remark}
Example \ref{ex1} above shows also that the assumption $\dim \L_{n}(\SS) = \dim \L_{n+1}(\SS) = \ldots = \dim \L_{2n}(\SS)$ is indispensable in the case $n=2$.  For bigger $n$ we can extend this example in the following way in order to show that it is impossible to shorten the chain of equalities in the conditions of Proposition \ref{pr_1} %this sequence 
even by 1.
\end{remark}

\begin{example} \label{ex2} Let $\A $ be generated by $
1, e_1, e_2, \ldots, e_n,e_{n+1}$ with the multiplication rules \[  e_1^2=e_2, \, e_1 e_{i}=e_{i+1},\ e_ie_1=0,\  i=2,\ldots , n-1, \ e_n^2=e_{n+1},\]  \[e_{n+1}^2=0, \ e_ie_j=0 \mbox{ for all pairs } i,j\in \{2,\ldots,n+1\}, i\ne j.\] Then $\SS=\{e_1\}$ is a generating system and we have \[\L_1(\SS)=\langle 1,e_1\rangle, \  \L_2(\SS)=\langle 1,e_1,e_2\rangle, \ \ldots, \ \L_n(\SS)=\langle 1,e_1,\ldots, e_n\rangle,\] further, $\L_{2n-1}(\SS)=\L_{2n-2}(\SS)=\ldots = \L_n(\SS)$, but \[\L_{2n}(\SS)=\langle 1,e_1,\ldots, e_{n+1}\rangle\ne \L_n(\SS).\]
\end{example}

For an arbitrary not necessarily associative algebra of dimension $n$ over $\F$ we can achieve a sharp bound of length, as will be proven below. The key element of the proof of this bound is the concept of a {\em fresh word}.

\begin{definition}
	A word $w$ of the length $n$ from generating set $\SS$ of algebra $\A$ is a {\em fresh} word, if for all integer $m,\ 0 \leq m<n,$ it holds  that $w \notin L_m(\SS)$.
\end{definition}

\begin{lemma}\label{lem_1}
	A fresh word of the  length greater than 1 is a product of two fresh words of non-zero lengths.
\end{lemma}
\begin{proof}
	Let us consider a word $w$ of the  length greater than 1. It can be represented as a product of  two words $s$ and $t$ of the lengths $a,b>0$, respectively. Let us assume that $s$ is not fresh. 
	
Then there exists $ a' \in \NN \cup \{0\}$, which satisfies $a'<a$ and $s \in \L_{a'} (\SS)$. Hence, $w$ which is  a word of the length $a+b$, belongs to $\L_{a'+b}(\SS)$ and, by the definition, it is not fresh. It can be shown in the same way that if $t$ is not fresh, then $w$ is not fresh as well, which proves the statement of the lemma.
\end{proof}

\section{Characteristic sequences and their basic properties}

In this section we introduce our main tool actual for all further considerations.

\begin{definition} \label{CharSeq}
Consider a unital $\F$-algebra $\A$   of the  dimension $\dim \A = n$,
	and its generating set $\SS$. By the {\em characteristic sequence} of $\SS$
	in $\A$ we understand a monotonically non-decreasing sequence of natural numbers $(m_0, m_1,\ldots, m_N)$, constructed by the following rules:
	\begin{enumerate}
		\item $m_0 = 0$.
\item Denoting $s_1=\dim \L_1(\SS) - 1$, we define $m_1=\ldots =m_{s_1}=1$.
\item If $m_0,\ldots, m_r$ are already constructed and the sets $ \L_1(\SS) ,\ldots,  \L_{k-1}(\SS) $ are considered, then we inductively continue the process in the following way. Denote $s_k=\dim \L_k(\SS) - \dim \L_{k-1}(\SS)$. Then $m_{r+1} =\ldots =m_{r+s_k}=k$.  
	\end{enumerate}
\end{definition}

\begin{remark} In other words, to construct a characteristic sequence we start with $m_0 = 0$ and for each $k=0,\ldots, l(\A)$ we add $(\dim \L_k(\SS) - \dim \L_{k-1}(\SS))$ elements equal to $k$.
\end{remark}

\begin{remark}
It is worth noting that in the associative case %adjacent 
 subsequent elements of a characteristic sequence are either equal or differ by 1, since for a generating set $\SS$, $\dim \L_k(\SS) - \dim \L_{k-1}(\SS) = 0$ implies that for every integer $h>k$, $\dim \L_h(\SS) - \dim \L_{h-1}(\SS) = 0$.
\end{remark}

\begin{lemma}\label{fr_char}
Let $\A$ be an $\F$-algebra, $\dim \A = n >2$, and $\SS$ be
a generating set for $\A$. Then 

1. Positive integer $k$ appears in the characteristic sequence  as many times as  many there are linearly independent fresh words of the length~$k$.

2. For any term $m_h$ of the characteristic sequence of $\SS$ there is a fresh word in $\L(\SS)$ of the length $m_h$.

3.	If there is a  fresh word in $\L(\SS)$ of the length $k$, then $k$ is included into the characteristic sequence of $\SS$.
\end{lemma}
\begin{proof}
1. Fresh words of  lengths less than  or equal to $k$ form a basis of $\L_k(\SS)$, therefore the number of fresh words of the  length exactly  $k$ is equal to $\dim \L_k(\SS) - \dim \L_{k-1}(\SS)$.

2. Follows directly from 1.

3. Follows from the proof of 1.
\end{proof}

\begin{lemma}\label{N=n-1}
Let $\A$ be an $\F$-algebra, $\dim \A = n>2$, and $\SS$ be
	a generating set for $\A$. Then the  characteristic sequence of $\SS$ contains $n$ terms, i.e., $N=n-1$. Moreover, $m_N=l(\SS)$.
\end{lemma}
\begin{proof}
By the definition for each $k=1,\ldots, l(\SS)$ on $k$th step  we add $(\dim \L_k(\SS) - \dim \L_{k-1}(\SS))$ terms to the characteristic sequence. Hence, the total number of terms is 
\[ 1 +  (\dim \L_1(\SS) - 1) + \ldots + (\dim \L_k(\SS) - \dim \L_{k-1}(\SS)) +\ldots \] \[ \ldots+ (\dim \L_{l(\SS)}(\SS) - \dim \L_{l(\SS)-1}(\SS)) =\dim \L(\SS)=n\]
since $\SS$ is a generating set. Also, by Definition \ref{sys_len}, the maximal $k$ such that $\dim \L_k(\SS) - \dim \L_{k-1}(\SS) >0$ is $l(\SS)$, hence, by Definition \ref{CharSeq}, we obtain $m_N=l(\SS)$.
\end{proof}

\begin{lemma}\label{sp_prop}
Let $\A$ be an $\F$-algebra, $\dim \A = n>2$. Assume, $\SS$  is
a generating set of $\A$, and $(m_0, m_1,\ldots, m_{n-1})$ is a characteristic sequence of $\SS$. Then  for
any integer $k \geq 0$ it holds that %the following holds:  
$\dim \L_k(\SS) = \max \{t|m_t \leq k\}+1$.
\end{lemma}
\begin{proof}
	We use the  induction on $k$.

 Induction base.  For $k=0$ the statement is trivial.
	
Induction step. Let us assume that the statement is true for $k=q$. Then for $k=q+1$ one has
$\dim \L_{q+1}(\SS)= (\dim \L_{q+1}(\SS) - \dim \L_{q}(\SS)) + \dim \L_{q}(\SS)$. By Definition~\ref{CharSeq} the summand $(\dim \L_{q+1}(\SS) - \dim \L_{q}(\SS))$ equals to the number $N_0$ of terms $(q+1)$ in the characteristic sequence. 
By the induction hypothesis 
$N_1=\dim \L_q(\SS) = \max \{t|m_t \leq q\}+1$, i.e., the increased by 1  index of the  last position in which $m_t \leq q$. 
%It is similar to the proof of Lemma \ref{N=n-1} that  $\dim \L_{q}(\SS))$ equals  to the increased by 1  index of the  last position in which $m_t \leq q$. 
By Definition~\ref{CharSeq} the sum $N_0+N_1$ %of these two numbers 
equals to the  increased by 1  index of the last position in which  $m_t \leq q+1$, or $\max \{t|m_t \leq q+1\}+1$.
\end{proof}

\begin{proposition}\label{pr_3_0}
Let $\A$ be an $\F$-algebra, $\dim \A = n>2$. Assume, $\SS$ is a generating set for $\A$ and $(m_0, m_1,\ldots, m_{n-1})$ is  the characteristic sequence of $\SS$. Then  for each $h$
	satisfying $m_h \geq 2$ it holds that there are indices $0<t_1 \leq t_2 < h$ such that $m_h=m_{t_1}+m_{t_2}$.
\end{proposition}
\begin{proof}
By Lemma \ref{fr_char} Item 1 each term $m_h$ of the characteristic sequence corresponds to a fresh word of the length $m_h$, denote it by $w_{m_h}$. By Lemma \ref{lem_1}, each fresh word of the length $m_h \ge 2$ can be represented as a product of two fresh words, possibly equal, of lesser lengths. Thus, $w_{m_h}=w_{k_1}\cdot w_{k_2}$ for some  fresh words $w_{k_1},  w_{k_2}$ of the  lengths $k_1 , k_2$, correspondingly.  Assume $k_1\le k_2$. Then by  Lemma \ref{fr_char} Item 3 there are indices $0<t_1 \leq t_2 < h$ such that $m_{t_1}=k_1$ and $m_{t_2}=k_2$. %So, $w_{m_h}=w_{m_{t_1}}\cdot w_{m_{t_2}}$. By 
Assume $k_1 > k_2$. Then by Lemma \ref{fr_char} Item 3 there are indices $0<t_1\le t_2<h$ such that $m_{t_1} = k_2$ and $m_{t_2}=k_1$. In both cases, the additivity of word length	concludes the proof.
\end{proof}

\section{Upper bound for the lengths of non-associative algebras}

\begin{theorem}\label{th_0}
Let $\A$ be an $\F$-algebra of the dimension $\dim \A = n$, $n>2$, $\SS$ be
a generating set for $\A$, $(m_0, m_1,\ldots, m_{n-1})$ be the characteristic sequence of $\SS$. Then for each  positive integer  $h \leq n-1$ it holds that $m_h \leq 2^{h-1}$.
\end{theorem}	
\begin{proof}
	We  prove this statement using the induction on $h$. 

The base. Case $h=1$ is trivial since $m_1=1\leq2^0$. 

The step. Let us assume  that for all positive integers $k$ such that $h \leq k <n-1$ the statement holds. We have to prove it now for $h=k+1\le n-1$. By Proposition \ref{pr_3_0} we have $m_{k+1}=m_{t_1}+m_{t_2}$, where $0<t_1 \leq t_2 < k+1$. According to the induction hypothesis, \[m_{t_1}+m_{t_2} \leq 2^{t_1 -1} +  2^{t_2 -1} \leq 2^{k-1}+2^{k-1}= 2^k,\] which concludes the proof.
\end{proof}

\begin{proposition}\label{gen_th}
Let $\A$ be an  $\F$-algebra,  $\dim \A = n>2$. Then $l( \A)\le 2^{n-2}$.
\end{proposition}
\begin{proof}
Let $\SS$ be an arbitrary generating set of $\A$. By Lemma \ref{N=n-1} the length $l(\SS)$ is equal to the last element of characteristic sequence of $\SS$. The index of this element is $\dim \A-1=n-1$. Hence by Theorem~\ref{th_0} we get 
$l(\SS) \leq 2^{(n-1)-1}=2^{n-2}$.
\end{proof}

The example below demonstrates that the obtained bound is sharp.

\begin{example}\label{ex2_0}
Let us consider an arbitrary field $\F$ and non-associative $\F$-algebra $\A$ of the dimension $n>2$ with the  basis $\{e_0=1 , e_1, \ldots, e_{n-1}\}$ and following multiplication rules: for every $k$ such that $1 \leq k \leq n-2$ we take 
$e_k ^2=e_{k+1},$
$e_{n-1}^2=0,$
	and for all 
$p,q$, $ p\ne q$, $1 \leq p,q \leq n-1, $
$e_p e_q = 0.$
	
Then the set $\SS$=\{$e_1$\} generates the algebra $\A$. Its characteristic sequence
is exactly $(0,1,2, \ldots ,2^{n-3}, 2^{n-2})$, since each new fresh word, except the first one, is the square
of the previous one. Then by Lemma \ref{N=n-1} we have $l(\SS)=2^{n-2}$. Since any generating set of $\A$ should contain $e_1$, we get $l(\A)=l(\SS)=2^{n-2}$. Since by the previous proposition $l(\A)\le 2^{n-2}$ and in general $l(\A) \ge l(\SS)$ we get $l(\A)=l(\SS)=2^{n-2}$.
\end{example}

\section{Basic properties of locally-complex algebras and their length}

The class of locally-complex algebras  provides a  natural generalization of the field of complex numbers $\C$, namely
\begin{definition}
	$\A$ is a {\em locally-complex algebra}, if it is finitely generated non-associative algebra over the field $\R$, such that any 1-generated subalgebra of $\A$, which is generated by an element of $\A \backslash \R$, is isomorphic to the field of complex numbers, $\C$.
\end{definition}

These algebras were introduced and investigated in \cite{BrSeSp11}. In this work we deal with the following equivalent definitions of locally-complex algebras, established in~\cite{BrSeSp11}.

\begin{lemma}[{\rm \cite[Lemma 4.1]{BrSeSp11}}] The following conditions are equivalent for a real unital algebra $\A$:
	\begin{enumerate}
		\item $\A$ is locally-complex;
		\item every $0\neq a \in A$ has a multiplicative inverse lying in $\R a + \R$ ;
		\item $\A$ is quadratic and $\A$ has no  nontrivial idempotents or square-zero elements;
		\item $\A$ is quadratic and $n(a)>0$ for every $0\neq a \in A$. 	\end{enumerate}
Moreover, if $2\le \dim A= n<\infty$, then (1)-(4) are equivalent to 
	\begin{enumerate}		
		\item[5.] $\A$ has a basis $\{1,e_1,\ldots,e_{n-1}\}$ such that $e_i^2=-1$ for all $i$ and $e_i e_j = -e_j e_i$ for all $i\neq j$.
	\end{enumerate}
\end{lemma}

Here by a quadratic algebra we understand such a unital $\R$-algebra $\A$ that for every $a \in \A$ the elements $1,a$ and $a^2$ are linearly dependent. By $n(a)$ we understand $a^2$ if $a\in \R$ and  a real number $n$ such that $a^2 -t(a) a +n=0$ ($t(a) \in \R$) if $a\not\in \R$. Since we mainly use multiplication tables as a way of defining algebras, the property (v) is the one which is the most relevant for the present paper.

For locally-complex algebras we can improve the bound obtained in Proposition~\ref{pr_1}.

\begin{proposition}\label{pr_2}
	Let us consider a finite subset $\SS$ of a locally-complex algebra $\A$ and an integer $n \geq 2$. Let \\
	$(i)$  $\dim \L_{n-1}(\SS)+1 = \dim \L_{n}(\SS)$, and \\ $(ii)$	
	$\dim \L_{n}(\SS) = \dim \L_{n+1}(\SS) = \ldots = \dim \L_{2n-1}(\SS).$ 
	\\
	Then for all $t \in \mathbb{N}$  it holds  that $\dim \L_{n}(\SS) = \dim \L_{n+t}(\SS)$.
\end{proposition}

\begin{proof}
	By Proposition \ref{pr_1}, it is sufficient to show that in these conditions $\dim \L_{n}(\SS) = \dim \L_{2n}(\SS)$.
	
	Let us consider $w$, a word of the length $2n$. It can be represented as a product $w=s t$ of two words $s,t$ such that   $l(s)=k>0$ and $l(t)=2n-k>0$. Without loss of generality $k \leq n$. If $k<n$, then, since $t$ is an element of $\L_{2n-k}(\SS) = \L_{n}(\SS)$, $w$ belongs to $\L_{n+k}(\SS)=\L_{n}(\SS)$. If $k=n$, then $s$ and $t$ can be represented as $s=s_0 + r_s c$ and $t=t_0 + r_t c$ respectively, where $s_0,t_0 \in \L_{n-1}(\SS)$, $r_s, r_t \in \R$ and $c$ is an element of $ \L_{n}(\SS) \setminus\L_{n-1}(\SS)$. This representation is correct since $\dim \L_{n-1}(\SS)+1 = \dim \L_{n}(\SS)$, and, hence, $c$ with the basis of $\L_{n-1}(\SS)$ composes the basis of $\L_{n}(\SS)$.
	
	It follows that  $$w=s   t=(s_0 + r_s c)(t_0 + r_t c)=s_0t_0 +(r_s c) t_0+ s_0 (r_t c) + (r_s c) (r_t c) $$
	
	The coefficients  $r_t$ and $r_s$ are real numbers, which allows us to omit the brackets and rearrange the order of multiplication. Thus, we get	
	$$w=s_0t_0+r_s c t_0+ r_t s_0 c + r_s r_t c^2 .$$
	
	Let us consider each term separately. 
	\begin{enumerate}
		\item[] $s_0t_0 \in \L_{2n-2}(\SS)$ as a product of two elements of $\L_{n-1}(\SS)$. 
		\item $r_s c t_0$ and $r_t s_0 c$ are elements of $\L_{2n-1}(\SS)$ as products of a real number, an element of $\L_{n-1}(\SS)$ and an element of $\L_{n}(\SS)$.
		\item  $r_s r_t c^2$ is an element of $\L_{n}(\SS)$, since  $c^2 = u + vc$, where $u, v \in \R$, as $\A$ is locally-complex.
	\end{enumerate}

	Since both $\L_{2n-2}(\SS)$ and $\L_{2n-1}(\SS)$ are equal to $\L_{n}(\SS)$ by the conditions,  it follows that $w$ is an element of $\L_{n}(\SS)$. This implies $\L_{2n}(\SS)=\L_{n}(\SS)$.
\end{proof}

Next example shows that to improve  Proposition \ref{pr_1} both additional conditions, that  $\A$ is locally-complex and condition $(i)$, are necessary.

\begin{example} Let $\A $ be generated by $
	1, e_1, e_2, \ldots, e_6$ with the multiplication given by
	\[e_1e_2=e_4=-e_2e_1,\ e_1e_3=e_5=-e_3e_1,\ e_4e_5=e_6=-e_5e_4, \ e_i^2=-1, \ i=1,\ldots,6,\]
	and all other products are zero. Then for $\SS=\{e_1,e_2,e_3\}$ we have that $\dim \L_1(\SS)=4$, $\dim \L_2(\SS)=6=\dim L_3(\SS)$, but $\dim \L_4(\SS)=7$.
	
	For $n > 2$ we consider the   algebra $\A$  generated by
	$1, e_1, e_2, \ldots, e_{n+3}$ with the multiplication rules
	$e_1e_2=e_3=-e_2e_1,\ e_1e_3=e_4=-e_3e_1,\ \ldots, $ $$ e_1e_{n-1}=e_n=-e_{n-1}e_1,\ e_1 e_n=e_{n+1}=-e_n e_1,\ 
	e_2 e_n=e_{n+2}=-e_n e_2,$$ $  e_{n+1}e_{n+2}=e_{n+3}=-e_{n+2}e_{n+1},\  e_i^2=-1,\ i=1,\ldots,n+3,$
	and all other products are zero. Then for $\SS=\{e_1,e_2\}$ we have $\dim \L_{n-1}(\SS)=n+1$, $\dim \L_{n}(\SS)=\ldots=\dim \L_{2n-1}(\SS)=n+3$, but $\dim \L_{2n}(\SS)=n+4$. So, comparing with Proposition \ref{pr_2} we see that $(i)$ is not satisfied, $(ii)$ is satisfied, and the result does not hold. It is straightforward to see that the introduced algebra $\A$ is locally-complex.
	
Example \ref{ex2} shows that conditions $(i)$ and $(ii)$ are not sufficient if $\A$ is not locally complex.	
\end{example}

\begin{corollary}\label{cor_1}
	Let $\SS$ be a finite generating set of a locally-complex algebra $\A$ and  $n \geq 2$ be  an integer. Assume that  $\dim \L_{n-1}(\SS) < \dim \L_{n}(\SS)$ and  $\dim \L_{n-1}(\SS) \leq \dim \A -2$. Then $\dim \L_{n-1}(\SS) \leq \dim \L_{2n-1}(\SS) -2$.
\end{corollary}

\begin{proof}
	Assume the opposite that $\dim \L_{n-1}(\SS) > \dim \L_{2n-1}(\SS) -2$. Then	there are the following  possibilities for $\dim \L_{n-1}(\SS)$:

1. $\dim \L_{n-1}(\SS)< \dim \L_{n}(\SS)-1$, or, in other words, $\dim \L_{n-1}(\SS) \leq \dim \L_{n}(\SS)-2$. Since $\dim \L_{n}(\SS) \leq \dim \L_{2n-1}(\SS)$, we get $\dim \L_{n-1}(\SS) \leq \dim \L_{2n-1}(\SS) -2$.
		
2. $\dim \L_{n-1}(\SS)= \dim \L_{n}(\SS)-1$. By Proposition \ref{pr_2}, if $\dim \L_{2n+1}(\SS) = \dim \L_{n}(\SS)$, then $\L(\SS)=\L_{n}(\SS)$. On the other hand, $$\dim \L_{n}(\SS)=\dim \L_{n-1}(\SS)+1 \leq \dim \A-2+1<\dim \A.$$ However, this contradicts to Remark \ref{rem}, namely, $\L(\SS)=\A$. Thus our assumption is wrong and $\dim \L_{2n+1}(\SS) \geq \dim \L_{n}(\SS)+1 = \dim \L_{n-1}(\SS)+2$.
\end{proof}

\begin{proposition}\label{pr_3}
	Let $\A$ be a locally-complex algebra of the dimension $\dim \A = n$, $n>2$. Assume $\SS$ is a generating set for $\A$ and $(m_0, m_1,\ldots, m_{n-1})$ is  the characteristic sequence of $\SS$. Then  for each $h$
	satisfying $m_h \geq 2$ it holds that there are indices $0<t_1 < t_2 < h$ such that $m_h=m_{t_1}+m_{t_2}$.
\end{proposition}
\begin{proof}
	By Lemma \ref{fr_char} Item 1 each term $m_h$ of the characteristic sequence corresponds to a fresh word of the length $m_h$, denote it by $w_{m_h}$. By Lemma \ref{lem_1}, each fresh word of the length $m_h \ge 2$ can be represented as a product of two fresh words, possibly equal, of lesser lengths. Thus, $w_{m_h}=w_{k_1}\cdot w_{k_2}$ for some  fresh words $w_{k_1},  w_{k_2}$ of the  lengths $k_1 , k_2$, correspondingly.  We consider three cases separately. 
	
	1. Assume $k_1 < k_2$. Then by  Lemma \ref{fr_char} Item 3 there are indices $0<t_1 < t_2 < h$ such that $m_{t_1}=k_1$ and $m_{t_2}=k_2$. %So, $w_{m_h}=w_{m_{t_1}}\cdot w_{m_{t_2}}$. By 
	
	2. Assume $k_1 > k_2$. Then by Lemma \ref{fr_char} Item 3 there are indices $0<t_1 < t_2<h$ such that $m_{t_1} = k_2$ and $m_{t_2}=k_1$. 
	
	3. Assume $k_1 = k_2$. If $w_{k_1}$ and $w_{k_2}$ are not linearly independent modulo $\R$, then there exist real numbers $r_0,r_1,r_2$ such that $r_0 +r_1 w_{k_1} +r_2 w_{k_2}=0$ and $r_0^2 +r_1^2 +r_2^2 \neq0$. Obviously, at least one of $r_1$ and $r_2$ is non-zero. Let us assume that it is $r_2$, the second case is similar. It follows that $w_{k_2} = r'_0 +r'_1 w_{k_1}$, where $r'_0 = -r_0/r_2,r'_1=-r_1/r_2$. Hence, \[w_{m_h}=w_{k_1} w_{k_2}=w_{k_1}   (r'_0 +r'_1 w_{k_1}) =r'_0 w_{k_1}  +r'_1 w_{k_1}^2. \] Note that since $\A$ is locally-complex, $w_{k_1}^2 \in \langle 1, w_{k_1} \rangle $, thus $w_{m_h} \in \langle 1, w_{k_1} \rangle$. However, this contradicts to the fact that $w_{m_h}$ is fresh. Thus, by Lemma \ref{fr_char} Items 1 and 3, there are at least two distinct indices $t_1$ and $t_2$ such that  $m_{t_1}=m_{t_2}=k_1=k_2$. 
	
	In all cases, the additivity of word length concludes the proof.
\end{proof}

\begin{remark} Observe that we proved that characteristic sequences of locally complex algebras belong to the class of additive chain without doubling, see \cite{Knut}, i.e., each term is a sum of different previous terms.
	\end{remark}
  
This small difference with Proposition \ref{pr_3_0} allows us to improve the upper bound on the length function established in Theorem \ref{th_0} for general algebras in the case of locally-complex algebras.

\section {Upper bound for the lengths of locally-complex algebras}

\begin{definition}
	Let ${\mathcal F}_n=(F_1,\ldots,F_n,\ldots)$ denote the Fibonacci sequence, i.e. the sequence of positive integers  satisfying the recurrent relations $F_1=F_2=1$, $F_{i}=F_{i-1}+F_{i-2}$ for all $i\ge 3$.
\end{definition}

\begin{theorem}\label{th_1}
Let $\A$ be a locally-complex algebra of the dimension $\dim \A = n$, $n>2$. Assume $\SS$ is
a generating set of $\A$, and $(m_0, m_1,\ldots, m_{n-1})$ is the characteristic sequence of $\SS$. Then for each positive integer $h \leq n-1$ it holds  that $m_h \leq F_h$.
\end{theorem}	
\begin{proof}
	We prove this statement using the induction  on~$h$. 
	
	The base. If $h=1$,  then $m_1=1\leq F_1$. 
	
	The step. Assume  that for all positive integers $h=1,\ldots,k$ the statement holds. We have to prove it now for $h=k+1\le n-1$. By Proposition \ref{pr_3} $m_{k+1}=m_{t_1}+m_{t_2}$, where $0<t_1 < t_2 < k+1$. According to the induction hypothesis, $$m_{t_1}+m_{t_2} \leq F_{t_1}+F_{t_2} \leq F_{k-1} + F_k= F_{k+1},$$ which concludes the proof.
\end{proof}

Now we can  prove our main result.

\begin{theorem}\label{key_th}
	Let $\A$ be a locally-complex algebra of the dimension $\dim \A = n >2$. Then the length of $\A$ is less than or equal to the $(n-1)$-th Fibonacci number $F_{n-1}$.
\end{theorem}

\begin{proof}
	Let $\SS$ be an  arbitrary generating set of $\A$, and  $(m_0, m_1,\ldots, m_{n-1})$ be its characteristic sequence. By Lemma \ref{N=n-1} and Theorem \ref{th_1}, $l(\SS)=m_{n-1} \leq F_{n-1}$. Hence, length of $\A$ is less than or equal to $F_{n-1}$.
\end{proof}

\begin{proposition}
	If $\A$ is a locally-complex algebra of the dimension $\dim \A = n$ and $\SS$ is its generating set, containing $k$ linearly independent modulo $\R$ elements, then
	$l(\SS) \leq F_{n-k+1}.$
\end{proposition}

\begin{proof}
	Let $(m_0,\ldots,m_{n-1})$ be the characteristic sequence of $\SS$. It should be noted that
	$m_1 = \ldots = m_k = 1$, since $\dim \L_1(\SS) - \dim \L_0(\SS) = k$. We   use  the induction to prove  that  $m_{k+h} \leq F_{h+2}$ for  all integer $h, \ -1 \le h \leq n-k-1$.
	
	The base. For $h=-1$ and $h=0$ one has  $m_{k-1}=m_k=1=F_1=F_2$.
	
	The step. Let us assume that for $h =-1,0,\ldots,d $ the statement holds. We have to prove now that it holds for $h=d+1 \le n-k-1$. By Proposition \ref{pr_3} we have $m_{k+d+1}=m_{t_1}+m_{t_2}$, where $0<t_1 < t_2 < k+d+1$. According to the induction hypothesis, $$m_{t_1}+m_{t_2} \leq F_{t_1-k+2}+F_{t_2-k+2} \leq F_{d+1} + F_{d+2}= F_{d+3},$$ which concludes the proof.
\end{proof}

%We will demonstrate below that achieved boundary cannot be improved.
The example below demonstrates that the obtained bound is  sharp in the class of locally-complex algebras.  

\begin{example}
	Let us consider locally-complex algebra $\A$ over real numbers with basis \{$e_0=1_{\R},..e_{n-1}$\} ($n>2$) and following multiplication rule: for every $k$, such that $1 \leq k \leq n-3$
	\[e_k e_{k+1}=e_{k+2}\]
	\[e_{k+1} e_k= - e_{k+2},\]
	for every $m$, such that $1 \leq m \leq n-1$
	\[e_m e_m = -1,\]
	and for other combinations of $p,q$: $1 \leq p,q \leq n-1$
	\[e_p e_q = 0.\]
	
The set $\SS$=\{$e_1,e_2$\} generates $\A$, and its characteristic sequence is exactly $(0,1,1,2, \ldots, F_{n-1})$, since every fresh word is obtained as a product of two previous fresh words. We get $F_{n-1}= l(\SS) \le l(\A) \le F_{n-1}$, which means $ l(\A) = F_{n-1}$.
\end{example}

\end{document}